\documentclass[12pt, reqno]{amsart}

\usepackage{amssymb,latexsym,amsmath,amsfonts}
\usepackage{mathrsfs}
\usepackage{graphicx}
\usepackage{upgreek}
\usepackage[usenames]{color}

\usepackage{hyperref}

\numberwithin{equation}{section}

\allowdisplaybreaks

\theoremstyle{definition}
\newtheorem{definition}{Definition}[section]

\theoremstyle{remark}

 \theoremstyle{plain}
\newtheorem{theorem}[definition]{Theorem}
\newtheorem{result}[definition]{Result}
\newtheorem{lemma}[definition]{Lemma}
\newtheorem{proposition}[definition]{Proposition}

\newcommand\cvx[1]{{\sf Conv({#1})}}
\newcommand\nwtp[1]{{\sf N({#1})}}
\newcommand\supp[1]{{\sf supp({#1})}}
\newcommand{\maf}[1]{\mathfrak{#1}}

\newcommand{\C}{\mathbb{C}} 
\newcommand{\R}{\mathbb{R}}

\makeatletter
\newcommand*{\rom}[1]{\expandafter\@slowromancap\romannumeral #1@}
\makeatother

\begin{document}

\title[Irreducible polynomial and Newton polytope]{Irreducibility of a sum of polynomials depending on disjoint sets of variables}

\author{Vikramjeet Singh Chandel}
\address{Harish-Chandra Research Institute, Prayagraj (Allahabad) 211019, India}
\email{vikramjeetchandel@hri.res.in}

\author{Uma Dayal}
\address{}
\email{umadayal@stanford.edu}

\thanks{Vikramjeet Singh Chandel was supported by an institute postdoctoral fellowship of IIT 
Bombay while working on this project}


\begin{abstract}
In this article, we give two different sufficient conditions for the irreducibility of a polynomial of more
than one variable, over an algebraically closed field, that can be written as a sum of two polynomials which depend on 
mutually disjoint sets of variables. These conditions are derived from analyzing the
Newton polytope of such a polynomial and then applying the `Irreducibility criterion' introduced 
by Gao.
\end{abstract}
\maketitle

\section{Introduction and statement of results}\label{S:intro}
In this article, we are interested in finding sufficient conditions that will guarantee
that a polynomial in several variables whose coefficients lie in an algebraically closed field $F$
is irreducible over the field $F$.
The method that we employ in our investigation is called the `Polytope Method' and is strongly 
motivated from the work of Gao
\cite{Gao01:AIrPolyNwtPol}.
Before we outline this method, let us denote by $\R$
the set of all real numbers.
\smallskip

Given a positive integer $n\geq 1$, and a nonempty subset $A\subset\R^n$, the {\em convex hull} of 
$A$, denoted by $\cvx{A}$, is the set defined by:
\[\cvx{A}:=\Big\{\sum_{1}^{m}t_j\,a_j:t_j\in[0,1] \ \text{with} \ \sum_{1}^{m}t_j=1 \ \text{and} \ 
a_j\in A, \ 1\leq j\leq m\Big\}.\]
A {\em polytope} in $\R^n$ is (by definition) the convex hull of finitely many points in $\R^n$. A point of a polytope is called a {\em vertex} if it does not lie
in the interior of the convex hull of two distinct points of the polytope. It is
a well known fact that a polytope is always the convex hull of its vertices.
We refer the reader to either one of the following: \cite{Ewald:CCAAG},
\cite{Grunbaum:CP}, \cite{Schneider:CB}, \cite{Ziegler:LP} for further basic
properties of polytope.
\smallskip

There is a very natural way of associating a polytope to a given polynomial.
In this article, we shall always consider polynomials in the variables $z_1,\ldots,z_n$ with
coefficients in an arbitrary but fixed algebraically closed field $F$. The set of all such
polynomials will be denoted by $F[z_1,\ldots,z_n]$. Given a polynomial 
$Q=\sum q_{i_1,\ldots,i_n}{z_1}^{i_1}\dots{z_n}^{i_n}$ belonging to $F[z_1,\ldots,z_n]$, the support of $Q$ denoted by $\supp{Q}$, is the set defined by
\[\supp{Q}:=\Big\{(i_1,\ldots,i_n):q_{i_1,\ldots,i_n}\neq 0,\,i_j\geq0\Big\}.\]
Here, if we write $I=(i_1,\ldots,i_n)$ then denote $|I|=\sum_{j=1}^ni_j$
and $z^{I}=z_1^{i_1}\dots z_n^{i_n}$.

\noindent
The {\em Newton polytope} of $Q$, denoted by $\nwtp{Q}$, is (by definition) the convex hull
of $\supp{Q}$. This is the association that we alluded to in the beginning of this paragraph.
\smallskip

The term `Polytope Method', as coined in the article \cite{Gao01:AIrPolyNwtPol},
has its origin in a paper by Ostrowski \cite{Ostrowski75:MFP}.
Ostrowski in \cite{Ostrowski75:MFP}
uses the term `Baric Polyhedron' in place of Newton polytope for a class of more general
polynomials called {\em algebraic polynomials} where the exponents of variables
are rational numbers. We also wish to refer the reader to the work done in the articles
\cite{Filaseta95:IBP}, \cite{Lipkovski88: NPhI} and \cite{Shanok36:CP&cI}
where methods based on studying the polytopes have been employed in determining the
irreducibility of polynomials.
\smallskip

We begin with a key result in the paper 
\cite[Theorem~\rom{6}]{Ostrowski75:MFP} for polynomials that is at the heart of
the `Polytope Method'. For this purpose, we need the notion of Minkowski
sum of convex sets.
\smallskip

\noindent Given convex sets $A$ and $B$ in $\R^n$, the Minkowski sum of $A$ and $B$, denoted
by $A+B$, is (by definition) the set $\big\{a+b: a\in A, \ b\in B\big\}$.
We now state:

\begin{result}[Ostrowski, \cite{Ostrowski75:MFP}]\label{Res:factpolyminksum}
Let $Q\in F[z_1,\ldots,z_n]$ be such that $Q=Q_1Q_2$ for some $Q_1,Q_2\in F[z_1,\ldots,
z_n]$. Then $\nwtp{Q}=\nwtp{Q_1} + \nwtp{Q_2}$.
\end{result}
\noindent
Based on Result~\ref{Res:factpolyminksum}, Gao in \cite{Gao01:AIrPolyNwtPol} gave an irreducibility criterion
by introducing the notion of an integrally indecomposable polytope. We shall state this criterion here but we need a few 
definitions.
\smallskip

A point in $\R^n$ will be called an {\em integral point} if all of its 
coordinates are integers. 
A polytope will be called an {\em integral polytope} if all of its vertices are integral points. Further,
an integral polytope $C$ is called {\em integrally decomposable} if there exist integral polytopes
$A$ and $B$, consisting of at least two points each, such that $C=A+B$. A polytope that is not 
integrally decomposable is called an {\em integrally indecomposable} polytope. Now we present 
the `Irreducibility criterion' due to Gao.
\smallskip

\noindent{\bf Irreducibility criterion.}
Let $Q\in F[z_1,\ldots,z_n]$ be a nonconstant polynomial that is not divisible by any of 
$z_i$. If the Newton polytope $\nwtp{Q}$ of $Q$ is integrally indecomposable then $Q$ is
irreducible over $F$.
\smallskip

Notice that the irreducibility criterion above follows very easily from 
Result~\ref{Res:factpolyminksum}.
Based on this criterion and by constructing integrally indecomposable polytopes, Gao in 
\cite{Gao01:AIrPolyNwtPol} gave new classes of irreducible polynomials. We present an important
result in \cite{Gao01:AIrPolyNwtPol} that characterizes 
integrally indecomposable prisms but first we define what a prism is. A {\em prism} is the convex hull of a set in 
$\R^n$ consisting of a polytope $C$, contained in a hyperplane $H$, and a point $v\in\R^n$ not belonging to the hyperplane $H$.
The point $v$ is called the {\em distinguished vertex} of the prism. Now we present

\begin{result}[{\cite[Theorem~4.2]{Gao01:AIrPolyNwtPol}}]\label{Res:indprism}
	Let $C$ be an integral polytope in $\R^n$ contained in some hyperplane 
	$H$ and let $v\in\R^n$ be an integral point that is not in $H$. Suppose that
	$v_1,\ldots,v_k$ are all the vertices of $C$. Let $\widetilde{C}=\{v\}\cup
	C$. Then the prism
	$\cvx{\widetilde{C}}$ is integrally indecomposable if and only if 
	\[\gcd(v-v_1,\ldots,v-v_k)=1.\]
\end{result}
\noindent Here, and elsewhere in this article, for an integral point $w$, $\gcd(w)$ shall denote the greatest common divisor of
the coordinates of $w$. For a finite set of integral points $w_1,\ldots,w_k$,
$\gcd(w_1,\ldots,w_k)$ will denote the greatest common divisor of the coordinates of $w_i$'s
taken together. 
\smallskip

Using Result~\ref{Res:indprism}, Gao constructed many classes of irreducible 
polynomials that were not known before. A particular class of polynomials for
which Gao gave a sufficient condition for irreducibility is the following.
Let $Q\in F[z_1,\ldots,z_n],n\geq 2$, be such that $Q(z)=Q_1(z_1)+Q_2(z_2,\ldots,z_n)$ where
$Q_1$ and $Q_2$ are nonconstant polynomials. Then $Q$ is irreducible if $\gcd(d(Q_1),\;d(Q_2))=1$, where
$d(Q_j),j=1,2$, denotes the degree of $Q_j$. This result motivates the following problem:
\begin{itemize}
	\item[$(*)$] Let $n\geq 2$, and consider a polynomial $P\in F[z_1,\dots,z_n], n\geq 2$, such that
	 $P(z_1,\dots,z_n)=P_1(z_1,\dots,z_{\nu})+P_2(z_{\nu+1},\dots,z_n)$, where $P_1$, $P_2$ are nonconstant polynomials
	and $1\leq\nu\leq n-1$. Investigate sufficient conditions under which $P$ is irreducible
	over $F$.
\end{itemize}
In the case $\nu=1$, one such condition is given by the above stated result of Gao.
By investigating the Newton polytope of a polynomial as stated in 
$(*)$, we find out that there is a special integral prism, as in Result~\ref{Res:indprism},
that is a {\em face} of the Newton polytope of 
such a polynomial. This fact enabled us to derive our first criterion which we
state here.
\begin{theorem}\label{T:irredCritdecpoly}
Let $P\in F[z_1,\ldots,z_n]$, $n\geq 2$, and $\nu$ be such that
$1\leq \nu \leq n-1$. Suppose we can write $P=P_1 + P_2$ where
$P_1\in F[z_1,\ldots,z_\nu], P_2\in F[z_{\nu+1},\ldots,z_n]$ are nonconstant polynomials.
Let $d(P_j)$ denote the degree of polynomials $P_j$, $j=1,2$.
Suppose that $\gcd(d(P_1),d(P_2))=1$, then $P$ is irreducible over $F$.
\end{theorem}
Clearly this result generalizes the result of Gao when $\nu=1$.
In Section~\ref{S:auxres}, where we do certain computations regarding the 
determination of $\nwtp{P}$, $P$ as in Theorem~\ref{T:irredCritdecpoly}, 
we shall also notice that, under a mild restriction on $P_j$, $\nwtp{P_j}$ are
faces of $\nwtp{P}$. This observation allows us to present our second criterion:

\begin{proposition}\label{P:anocritirred}
Let $P\in F[z_1,\ldots,z_n]$, $n\geq 2$. Supose we can write $P=P_1 + P_2$ where
$P_1\in F[z_1,\ldots,z_\nu], P_2\in F[z_{\nu+1},\ldots,z_n]$ are nonconstant polynomials.
If for some $i$, $1\leq i\leq 2$, we have $P_i(0)=0$ such that
none of $z_j$'s divide $P_i$ for any $j$ and $\nwtp{P_i}$ is integrally indecomposable.
Then $P$ is irreducible over $F$.
\end{proposition}

The proof of the above proposition is presented in Section~\ref{S:auxres} while
the proof of Theorem~\ref{T:irredCritdecpoly} is presented in Section~\ref{S:proofMT}.

\section{A few auxiliary results and proof of Proposition~\ref{P:anocritirred}}\label{S:auxres}
This section is devoted to the study of the Newton polytope $\nwtp{P}$ of the polynomial
$P$, where $P$ is as in the problem $(*)$. One of the important results of this section 
is a result which says that under the condition that $P_j(0)=0$, $j=1,2$, $\nwtp{P_j}$ are
faces of $\nwtp{P}$. This is Proposition~\ref{P:facesofpoly} below.
We start with recalling the definition of a face of a general convex set $C$.
\smallskip

Let $H$ be a hyperplane in $\R^n$. By its definition there exists a nonzero $\alpha\in\R^n$ and
$\gamma\in\R$ such that $H:=\big\{x\in\R^n:\sum_{j=1}^n\alpha_jx_j=\gamma\big\}$.
The hyperplane divides $\R^n$ into the following two half spaces:
\[
H^{+}:=\Big\{x\in\R^n:\sum_{1}^n\alpha_jx_j-\gamma\geq 0\Big\}, \ \text{and} \ 
H^{-}:=\Big\{x\in\R^n:\sum_{1}^n\alpha_jx_j-\gamma\leq 0\Big\}.
\]
A {\em supporting hyperplane} for a convex set $C$ is a hyperplane $H$ such that 
$H\cap C\neq\emptyset$  and either $C\subseteq H^{+}$ or $C\subseteq H^{-}$.
A {\em face} of $C$ is then a set of the form $C\cap H$, where $H$ is a supporting
hyperplane for $C$.
\smallskip

We begin with an elementary lemma that describes the convex hull of a union of two sets.
This lemma should be there in the literature; we present it here for the sake making the 
article to be more self contained.
\begin{lemma}\label{L:cvxhullunion}
Let $A,B$ be two nonempty finite sets in $\R^n$. Then
\[
\cvx{A\cup B}=\big\{t\alpha+(1-t)\beta : t\in[0,\,1], \ \alpha\in\cvx{A}, \ \beta\in\cvx{B}\big\}.
\] 
\end{lemma}
\begin{proof}
Let $x \in \cvx{A\cup B}$. Then there exist $x_i \in A\cup B$ and $t_i \in [0, 1]$, 
$1 \leq i \leq l$, such that
$x = \sum_{1}^{l} t_ix_i$ and $\sum_{1}^{l} t_i = 1$.
Consider the following sets
\[
I_A(x) = \big\{i:1\leq i\leq l, \ x_i \in A\setminus B\big\}, \quad I_B(x) =
\big\{i : 1\leq i\leq l, \ x_i \in B\big\}.
\]
Notice that $I_A(x) \cap I_B(x) = \emptyset$ and $I_A(x) \cup I_B(x) = \{1, \ldots, l\}$.
Therefore, we have $x = \sum_{i \in I_A(x)}t_ix_i + \sum_{i \in I_B(x)}t_ix_i$.
Assuming that $ \sum_{i \in I_A(x)}t_i \neq 0, 1$, we write
\[
x = \lambda \; \Big(\sum_{i \in I_A(x)}(t_i/\lambda)x_i\Big) + \mu \; \Big(
\sum_{i \in I_B(x)}(t_i/{\mu})x_i\Big),
\]
where $\lambda = \sum_{i \in I_A(x)}t_i$, $\mu = \sum_{i \in I_B(x)}t_i$.
Clearly $\lambda + \mu = 1$. If we set
\[
\alpha = \sum_{i \in I_A(x)}({t_i}/{\lambda})x_i, \quad \beta
 = \sum_{i \in I_B(x)}({t_i}/{\mu})x_i.
\]
then $\alpha \in \cvx{A}$ and $\beta \in \cvx{B}$, since $\sum_{i \in I_A(x)}
({t_i}/{\lambda}) = \sum_{i \in I_B(x)}({t_i}/{\mu}) = 1$.
Hence, $x = \lambda\alpha + (1 - \lambda)\beta$. We also notice if $\lambda=0$
or $\lambda=1$ then $x\in\cvx{B}$ or $x\in\cvx{A}$ respectively.
This establishes that $\cvx{A\cup B} \subseteq \big\{t\alpha+(1-t)\beta :
 t\in[0,\,1], \ \alpha\in\cvx{A}, \ \beta\in\cvx{B}\big\}$.
 \smallskip
 
To see the converse, let
$x = t\alpha + (1 - t)\beta$ for some $t \in [0, 1]$ and $\alpha \in \cvx{A}, \beta\in\cvx{B}$.
Now there exist $a_i \in A$, $\lambda_i \in [0, 1], 1 \leq i \leq l$ and $b_j \in B$,
$\gamma_j \in[0, 1], 1 \leq j \leq m$, such that
\[
\alpha =  \sum_{1}^{l}\lambda_ia_i, \ \sum_{1}^{l}\lambda_i = 1 \ \text{and}
 \  \beta =  \sum_{1}^{m}\gamma_jb_j, \ \sum_{1}^{m}\gamma_j = 1
\]
Therefore, $x = \sum_{1}^{l}t\lambda_ia_i + \sum_{1}^{m}(1 - t)\gamma_jb_j
\text{ and } \sum_{1}^{l}t\lambda_i+ \sum_{1}^{m}(1 - t)\gamma_j = 1$.
This proves that $x \in \cvx{A\cup B}$.
\end{proof}

We wish to compute the Newton polytope $\nwtp{P}$ with $P$ as in problem $(*)$.
Before we do this in our next lemma, we make the following
observations:
\begin{itemize}
\item[$(i)$] Given a nonconstant polynomial $Q\in F[z_1,\ldots,z_n]$, $Q$ is irreducible if
and only if 
$\widetilde{Q}(z_1,\ldots,z_n):=Q(z_1+a_1,\ldots,z_n+a_n)$, where $a_i$'s are any
elements of $F$, is irreducible.

\item[$(ii)$] Given a nonconstant polynomial $Q\in F[z_1,\ldots,z_n]$
such that $Q(0)\neq 0$, we know there exist
$a_i\in F, 1\leq i\leq n$ such that $Q(a_1,\ldots,a_n)=0$. Then 
$\widetilde{Q}(z_1,\ldots,z_n):=Q(z_1+a_1,\ldots,z_n+a_n)$ satisfies $\widetilde{Q}(0)=0$.
Notice that the fact that $F$ is algebraically closed is crucial here.
\end{itemize}

Because of $(i)$ and $(ii)$ above, we can assume, without loss of generality, when considering the irreducibility/reducibility of $P$, that 
$P_1(0)=0$ and $P_2(0)=0$. Now we compute $\nwtp{P}$.
\begin{lemma}\label{L:newtpdecompoly}
Let $n\geq 2$ and let $P\in F[z_1,\dots,z_n]$ be a polynomial such that $P=P_1 + P_2$
where $P_1\in F[z_1,\dots,z_{\nu}]$ and $P_2\in F[z_{\nu+1},\dots,z_n]$ be such that 
$P_1(0)=0, P_2(0)=0$. Then the Newton polytope of $P$ is given by:
\[
\nwtp{P}= \big\{t\alpha+(1-t)\beta : t\in[0,\,1], \ \alpha\in\nwtp{P_1},
\ \beta\in\nwtp{P_2}\big\}.
\]
\end{lemma}

\begin{proof}
The reader will discern that the above follows from Lemma~\ref{L:cvxhullunion},
once we establish that $\supp{P} = \supp{P_1} \cup \supp{P_2}$.
Notice that $\supp{P} \subseteq \supp{P_1} \cup \supp{P_2}$.
\smallskip

\noindent{\bf Claim.} $\supp{P_1}\cap \supp{P_2} = \emptyset$.

\noindent To see this let $\alpha \in \supp{P_1}\cap\supp{P_2}$. 
This implies $\alpha \in \supp{P_1}$ and $\alpha \in \supp{P_2}$. 
We know $\supp{P_1} \subseteq \big\{(z_1, z_2, \ldots, z_n) \in \C^n: z_{\nu + 1}
 = \cdots = z_n = 0\big\}$ and $\supp{P_2} \subseteq \big\{(z_1, z_2, \ldots, z_n)
 \in \C^n: z_{1} = \cdots = z_\nu = 0\big\}$ where $1 \leq \nu \leq n$.
Hence $\alpha = 0\in\R^n$. However, since $P_1(0) = 0$ and $P_2(0) = 0$,
this is a contradiction. Thus, $\supp{P_1}\cap \supp{P_2} = \emptyset$.
\smallskip

It follows from the claim above that $\supp{P_1}  \cup \supp{P_2} \subseteq \supp{P}$
since none of the monomials in $P_1$ can cancel out the monomials in $P_2$ and viceversa.
Hence we have $\supp{P_1}  \cup \supp{P_2} = \supp{P}$. The lemma itself now follows 
from Lemma~\ref{L:cvxhullunion}.
\end{proof}

Now we establish the result that was alluded to in the introduction of this section. 

\begin{proposition}\label{P:facesofpoly}
The Newton polytopes $\nwtp{P_1}$ and $\nwtp{P_2}$ associated to 
the polynomials $P_1$ and $P_2$, as in Lemma~\ref{L:newtpdecompoly}, are faces of $\nwtp{P}$.
\end{proposition}

\begin{proof}
For every $j\in\{1,\ldots, n\}$, let us consider
\[
H_j := \big\{(x_1, \ldots, x_n) \in \R^n:x_j=0\big\} \text{ and } H_j^{+} :=
\big\{x\in\R^n: x_j \geq 0\big\}.
\]
Notice that $\nwtp{P} \subset H_j^{+}$ for all $j\in\{1,\ldots, n\}$ and
$\nwtp{P}\cap H_j \neq \emptyset$ for every $j\in\{1,\ldots, n\}$.
Hence, from the definition of a face, $\nwtp{P} \cap H_j$ are faces of $\nwtp{P}$.
\smallskip

\noindent{\bf Claim.} $\nwtp{P_1} = \bigcap^{n}_{\nu + 1}\big(\nwtp{P} \cap H_j\big)$.\smallskip

\noindent{It is clear that $\nwtp{P_1} \subseteq \bigcap^{n}_{\nu + 1}
\big(\nwtp{P} \cap H_j\big)$ since $\nwtp{P_1} \subset H_j$ for every
$j\in\{\nu + 1, \ldots, n\}$. To see the converse, let 
$x\in\bigcap^{n}_{\nu + 1}\big(\nwtp{P} \cap H_j\big)$ then
$x\in\nwtp{P}\cap H_j$ for every $j\in\{\nu + 1, \ldots, n\}$.
This implies that there exist $t\in[0,1]$ and $\alpha\in\nwtp{P_1}, \beta\in\nwtp{P_2}$ such that
\[
x = t\alpha + (1 - t)\beta \text{ and } x_j = 0 \ \text{for all} \  j\geq\nu + 1.
\]
If $t=1$, we see that $x=\alpha\in\nwtp{P_1}$. So, let us suppose this is not the case, i.e.,
$t\neq1$. We have $x_j = t\alpha_j + (1 - t)\beta_j$ and $\alpha_j = 0$ for every
$j \geq \nu + 1$. Hence for each $j \geq \nu + 1$, $x_j = 0$ if and only if $(1 - t)\beta_j = 0$.
This implies that $\beta = 0\in\R^n$, which is a contradiction since $0\notin\nwtp{P_2}$.
So $t=1$ and $x\in\nwtp{P_1}$. Therefore the converse holds true and the claim
above is established.}

Thus $\nwtp{P_1}$ is the intersection of finitely many faces of $\nwtp{P}$. 
It is a fact (\cite[Lemma~4.5, p.15]{Ewald:CCAAG}) that the intersection of finitely many faces of a convex set is also a face.
Therefore, $\nwtp{P_1}$ is a face of $\nwtp{P}$.
Arguing in a similar fashion and working with hyperplanes $H_j$, $1\leq j\leq \nu$, we see that
$\nwtp{P_2}$ is also a face of $\nwtp{P}$.
\end{proof}

We shall now present a proof of Proposition~\ref{P:anocritirred}. Before that we need a result 
that says how faces of a polytope decompose under Minkowski sum. We shall also use this result in Section~\ref{S:proofMT}. The reader is referred to
\cite[Theorem~1.5, p.~105]{Ewald:CCAAG} for a proof of this.
\begin{result}\label{Res:minkdecface}
Let $A$ and $B$ be polytopes in $\R^n$ and suppose $C=A+B$. Then
every face of $C$ is a Minkowski sum of unique faces of $A$ and $B$.
\end{result}
We are now ready to present

\begin{proof}[Proof of Proposition~\ref{P:anocritirred}]
We shall assume for simplicity that $i=1$ in Proposition~\ref{P:anocritirred}.
Let $(b_{\nu+1},\ldots,b_n)$ be such that if we write
\[
\widetilde{P}(z_1,\ldots,z_n):=P(z_1,\ldots,z_\nu,z_{\nu+1}+b_{\nu+1},\ldots,z_n+b_n)
 =P_1 + \widetilde{P_2}
\]
where $\widetilde{P_2}(z_1,\ldots,z_n):=P_2(z_{\nu+1}+b_{\nu+1},\ldots,z_n+b_n)$, then 
$\widetilde{P_2}(0)=0$. Now, suppose $P$ is reducible over $F$ then so is $\widetilde{P}(z_1,\ldots,z_n)$. Let $Q_1$, $Q_2$ be two nonconstant 
polynomials such that $\widetilde{P}=Q_1\,Q_2$. By Result~\ref{Res:factpolyminksum} we have 
\[
\nwtp{\widetilde{P}}=\nwtp{Q_1}+\nwtp{Q_2}.
\]
By Proposition~\ref{P:facesofpoly}, we see that $\nwtp{P_1}$ is a 
face of $\nwtp{\widetilde{P}}$. Hence, Result~\ref{Res:minkdecface} implies that there exist faces $A_1$, $B_1$ of polytopes $\nwtp{Q_1}$, $\nwtp{Q_2}$ respectively such that
\begin{equation}\label{E:P1red}
\nwtp{P_1}=A_1 + B_1.
\end{equation}

\noindent{{\bf Claim.} Both $A_1$ and $B_1$ must contain at least two points.}
\smallskip

\noindent{We begin with the observation that each vertex of $A_1$, $B_1$ is a vertex of $\nwtp{Q_1}$, $\nwtp{Q_2}$
respectively (follows from Result~\ref{Res:transf}). We also recall 
that for a polynomial $P$, the vertex set of $\nwtp{P}$ is a subset of $\supp{P}$.
Now, suppose $A_1$ consists of only a single
point, i.e., $A_1=\{\alpha=(\alpha_1,\dots,\alpha_n):\alpha\neq\bf{0}\}$, where
each $\alpha_j$ is a positive integer. Then every vertex $v$ of $\nwtp{P_1}$ is given by $v=w+\alpha$, where $w$ is a vertex of $B_1$. This will imply that
$z^{\alpha}:=z^{\alpha_1}\dots z^{\alpha_n}$ divides ${P_1}$ which is a contradiction to our hypothesis. Hence $A_1$ must contain at least two 
points. Similarly, $B_1$ must contain at least two points.}
\smallskip

The above claim together with \eqref{E:P1red} implies that $\nwtp{P_1}$ is 
integrally decomposable which is a contradiction to our hypothesis. Therefore $\widetilde{P}$ and consequently $P$ must be irreducible over $F$. 
\end{proof}

\section{Proof of the main theorem}\label{S:proofMT}
In this section, we shall present the proof of our main theorem. In this direction, let $P$ be a polynomial in $F[z_1,\ldots,z_n]$ such that 
$P=P_1+P_2$, where $P_1\in F[z_1,\ldots,z_\nu], \ P_2\in F[z_{\nu+1},\ldots,z_n]$ are nonconstant polynomials with $P_j(0)=0$.
We shall first describe certain faces of the Newton polytope of $P$. For this
purpose we shall need a well known result about the geometry of polytopes. The 
result is:
\begin{result}\label{Res:transf}
	Let $C$ be a polytope and let $F_1$ be a face of $C$. Suppose $F_0$ is any
    face of $F_1$ then $F_0$ will be a face of $C$.
\end{result}
\noindent The reader is referred to \cite[Theorem~1.7, p.~31]{Ewald:CCAAG} for a proof of this result.
\smallskip

We now begin with describing certain faces of $\nwtp{P}$. Define:
\[
\mathfrak{A}_j:=\Big\{v\in\R^n : v \ \text{is a vertex of} \ \nwtp{P_j} \ \text{such that} \ 
\sum_{1}^{n}v_i = d(P_j)\Big\}.
\]
Here, $d(P_j)$, $j=1,2$, denotes the degree of $P_j$.
It is clear that $\mathfrak{A}_j\neq\emptyset$, $j=1,2$. We also consider the hyperplane defined by:
\begin{equation}
H_0:=\Big\{x\in\R^n : d(P_2)\Big(\sum_{1}^\nu x_i\Big) + d(P_1)\Big(\sum_{1}^{n-\nu} x_{\nu+j}\Big)
=d(P_1)d(P_2)\Big\}.\label{E:hypmain}
\end{equation}

\begin{proposition}\label{P:cyldrface}
The convex set $H_0\cap \nwtp{P}$ is a face of $\nwtp{P}$. Moreover,
\begin{equation}
H_0\cap \nwtp{P}=\Big\{t\alpha+(1-t)\beta : t\in [0, 1], \ \alpha\in\cvx{\mathfrak{A}_1},\ 
\beta\in\cvx{\mathfrak{A}_2}\Big\}.\label{E:cyldrface}
\end{equation}
\end{proposition}

\begin{proof}
Let $L(x):=d(P_2)\big(\sum_{1}^\nu x_i\big) + d(P_1)\big(\sum_{1}^{n-\nu} x_{\nu+j}\big)$. Then if we
let $X=tX_1+(1-t)X_2$, where $X_1\in\nwtp{P_1}, X_2\in\nwtp{P_2}$ and $t\in [0, 1]$, then $L(X)=
tL(X_1)+(1-t)L(X_2)$.
\smallskip

Let $\{w_1,\ldots,w_k\}$ be the set of all vertices of $\nwtp{P_1}$. Then
there exist $s_j\in[0,1]$, $1\leq j\leq k$, and $\sum_{1}^k s_j=1$ such that
we have:
\begin{align*}
X_1&=\sum_{1}^k s_jw_j=\sum_{1}^k s_j(w_{j1},\ldots,w_{j\nu},0,\ldots,0)\\
&=\Big(\sum_{1}^k s_jw_{j1},
\sum_{1}^k s_jw_{j2},\ldots,\sum_{1}^k s_jw_{j\nu},0,\ldots,0\Big).\
\end{align*}
Here, $w_{ji}$ denotes the $i$-th coordinate of the vertex $w_j$.
It follows then that
\begin{equation}
\sum_{i=1}^{\nu}X_{1i} \ = \ \sum_{i=1}^{\nu}\;\sum_{j=1}^k s_jw_{ji} \ = \ \sum_{j=1}^ks_j\Big(\sum_{i=1}^{\nu}w_{ji}\Big).\label{E:sum1X1}
\end{equation}
For each $j$, since $w_j$ is a vertex of $\nwtp{P_1}$, we have
$
\sum_{i=1}^{\nu} w_{ji}\leq d(P_1) \label{E:sumvP1}.
$
From this and \eqref{E:sum1X1} we have:
\[
\sum_{i=1}^{\nu} X_{1i} \ = \ \sum_{j=1}^ks_j\Big(\sum_{i=1}^{\nu}w_{ji}\Big) \ \leq \ \sum_{j=1}^k s_j d(P_1) \ = \ d(P_1).
\]
Hence $L(X_1)=d(P_2)\sum_{i=1}^{\nu}X_{1i}\leq d(P_2)d(P_1)$. This proves that $\nwtp{P_1}$ lies 
in the negative half space determined by $H_0$.
\smallskip

The inequality $L(X_1)\leq d(P_2)d(P_1)$ becomes an equality if and only if 
\begin{equation}
\sum_{i=1}^{\nu} X_{1i} \ = \ d(P_1) \ = \ \sum_{j=1}^ks_j\Big(\sum_{i=1}^\nu w_{ji}\Big) \ = \ d(P_1).
\label{E:vertA1}
\end{equation}
Define $\mathscr{A}:=\{j: 1\leq j\leq k: s_j\neq 0\}$. Then it follows from \eqref{E:vertA1} that for each
$j\in\mathscr{A}$, $w_j\in\mathfrak{A}_1$.
\smallskip
 
 Therefore, if $X_1\in\nwtp{P_1}$ then $L(X_1)\leq d(P_1)d(P_2)$, and this inequality becomes an equality if and only if 
 $X_1\in\cvx{\mathfrak{A}_1}$. Similarly, we shall have that if $X_2\in\nwtp{P_2}$ then $L(X_2)\leq d(P_1)d(P_2)$ and this inequality
 becomes an equality if and only if $X_2\in\cvx{\mathfrak{A}_2}$. From these two assertions it follows that
 $L(X)\leq d(P_1)d(P_2)$ and this inequality is an equality if and only if $X_j\in\cvx{\mathfrak{A}_j}$. This proves 
 that $H_0\cap\nwtp{P}$ is a face of $\nwtp{P}$ and is given by \eqref{E:cyldrface}.
\end{proof}

Let us consider the following sets:
\begin{align*}
\maf{A}_{11}&=\big\{v\in\maf{A}_1:v_1\geq w_1 \ \text{for any} \  w\in\maf{A}_1\big\}, \ \text{and} \\
\maf{A}_{1i}&=\big\{v\in\maf{A}_{1(i-1)}:v_{i}\geq w_{i} \ \text{for any} \  w\in\maf{A}_{1(i-1)}\big\} \ \text{for each $i$, $2\leq i\leq\nu$}.
\end{align*}
Notice that since $\maf{A}_1$ is nonempty, each of $\maf{A}_{1i}$ is nonempty.
Moreover we have
\[\maf{A}_1 \ \supseteq \ \maf{A}_{11} \ \supseteq\dots\supseteq \ \maf{A}_{1\nu}.\]
Notice that $\maf{A}_{1\nu}$ is a singleton set. This is because if $X_1,X_2
\in\maf{A}_{1\nu}$ then $X_{1i}=X_{2i}$ for all $i$, $1\leq i\leq \nu-1$.
Now since $\sum_{1}^{\nu} X_{1i}=\sum_{1}^{\nu} X_{2i}$, we have $X_{1\nu}=X_{2\nu}$. Hence $X_1=X_2$.
\smallskip

For each $i$, $1\leq i\leq\nu$, we consider 
\begin{align*}
C_i&=\Big\{t\alpha+(1-t)\beta: t\in[0,1], \ \alpha\in\cvx{\maf{A}_{1i}}, \ \beta\in\cvx{\maf{A}_2}\Big\}\\
H_i&=\Big\{x\in\R^n:d(P_2)x_i +\Big(\sum_{1}^{n-\nu}x_{\nu+j}\Big)v_i=v_id(P_2)\Big\}.
\end{align*}
Here, $v_i$ is the $i$-th coordinate of any vector $v\in\maf{A}_{1i}$.
Now we present two very crucial lemmas that we shall need in our proof of the 
main theorem.

\begin{lemma}\label{L:1stface}
	If we let $C_0=H_0\cap\nwtp{P}$ be as in the Proposition~\ref{P:cyldrface}.
	Then $C_0\cap H_1=C_1$ and $C_1$ is a face of $C_0$. 
\end{lemma}
\begin{proof}
	Let $\alpha\in C_0$ then there exist $X_1\in\cvx{\maf{A}_1}$, $X_2\in\cvx{\maf{A}_2}$ such that $\alpha=tX_1+(1-t)X_2$, $t\in[0,1]$.
	If we let
	$L_1(x)=d(P_2)x_1 + \big(\sum_{1}^{n-\nu}x_{\nu+j}\big)v_1$, then 
	$L_1(\alpha)=tL_1(X_1)+(1-t)L_1(X_2)$. Now $L_1(X_1)=d(P_2)X_{11}$ and 
	since $v_1\geq w_1$ for any $w\in\maf{A}_1$, we see that $v_1\geq X_{11}$
	with equality when $X_1$ is a convex combination of vertices belonging to
	$\maf{A}_1$ whose first co-ordinate is $v_1$.
	\smallskip
	
	Clearly $L_1(X_2)=v_1d(P_2)$. Hence we see that $tL_1(X_1)+(1-t)L_1(X_2)=
	td(P_2)X_{11}+(1-t)v_1d(P_2)\leq tv_1d(P_2)+(1-t)v_1d(P_2)=v_1d(P_2)$. The 
	inequality in here is an equality if and only if $X_1\in\cvx{\maf{A}_{11}}$.
	This establishes the lemma above.
\end{proof}

Our next lemma concludes that for each $i$, $1\leq i\leq \nu-1$, $C_{i+1}$ is a face of $C_i$.
\begin{lemma}\label{L:subseqface}
	For each $i$, $1\leq i\leq \nu-1$, $C_i\cap H_{i+1}$ is a face of $C_i$ and 
	is equal to $C_{i+1}$, where $H_{i+1}$ is the hyperplane as defined above.
\end{lemma}
\begin{proof}
	For each $i$, let us set $L_{i+1}(x)=d(P_2)x_{i+1}+\big(\sum_{1}^{n-\nu}x_{\nu+j}\big)v_{i+1}$.
	Let $\alpha\in C_i$ then $\alpha=tX_1+(1-t)X_2$ for some $t\in[0,1]$ and $X_1\in\cvx{\maf{A}_{1i}}, X_2\in \cvx{\maf{A}_2}$. Hence
	$L_{i+1}(\alpha)=tL_{i+1}(X_1)+(1-t)L_{i+1}(X_2)$ for every $i$.
	\smallskip
	
	Notice $L_{i+1}(X_1)=d(P_2)X_{1(i+1)}$. Since $X_1\in\cvx{\maf{A}_{1i}}$, we see that $X_{1(i+1)}\leq v_{i+1}$ with equality if and only if $X_1\in\cvx
	{\maf{A}_{1(i+1)}}$. On the other hand $L_{i+1}(X_2)=v_{i+1}d(P_2)$. The lemma now follows from similar arguments as in the last paragraph of the proof of the previous lemma.
\end{proof}

We are now ready to present the proof of our main theorem.

\subsection{Proof of Theorem~\ref{T:irredCritdecpoly}}
Observe that without loss of generality we can assume that 
$P_j(0)=0$, $j=1,2$. We notice that, since $\maf{A}_{1\nu}$ is a singleton set, $C_\nu$ is a prism with its distinguished vertex
being the unique element of $\maf{A}_{1\nu}$ and its base $\cvx{\maf{A}_2}$. From Lemma~\ref{L:subseqface},
we see that $C_\nu$ is a face of $C_{\nu-1}$. Applying Lemma~\ref{L:subseqface} and Result~\ref{Res:transf} 
iteratively we get that $C_\nu$ is a face of $C_1$
which, by Lemma~\ref{L:1stface}, is a face of $C_0$. Proposition~\ref{P:cyldrface} says that $C_0$
is a face of $\nwtp{P}$. Again applying Result~\ref{Res:transf}, we get that $C_\nu$ is a face of $N(P)$.
\smallskip

Let us denote by $X_0$ the unique element of $\maf{A}_{1\nu}$ and let $X_1$
be any vertex of $\maf{A}_2$. Then the segment $\big\{tX_0+(1-t)X_1:t\in[0,1]\big\}$
is an edge of the prism $C_\nu$. Since $C_\nu$ is a face of $\nwtp{P}$, using
Result~\ref{Res:transf} again, we get that the segment is also an edge of the polytope $\nwtp{P}$.
\smallskip

\noindent{\bf Claim.} The edge $\big\{tX_0+(1-t)X_1:t\in[0,1]\big\}$ is integrally indecomposable.
\smallskip

\noindent To see this, suppose the edge is not integrally indecomposable. Then,
by Result~\ref{Res:indprism}, we get that $\gcd({X_0-X_1})=\gcd(X_{01},\ldots,
X_{0\nu},-X_{1(\nu+1)},\ldots,-X_{1n})=r\neq 1$. This implies 
$r$ divides 
$\sum_{i=1}^{\nu}X_{0i}$ and $r$ divides
$\sum_{j=1}^{n-\nu},X_{i(\nu+j)}$. Since $\sum_{i=1}^{\nu}X_{0i}=d(P_1)$
and $\sum_{j=1}^{n-\nu}X_{i(\nu+j)}=d(P_2)$, we have $r$ divides $\gcd(d(P_1),d(P_2))$.
This gives a contradiction. Hence the edge is integrally indecomposable.
\smallskip

Now we claim that $\nwtp{P}$ is integrally indecomposable. This is because if it is not so then using
Result~\ref{Res:minkdecface} we see that the edge as described above will be integrally decomposable. Since none of the 
$z_i$'s divide $P$, from the `Irreducibility criterion' $P$ is irreducible. \qed

\subsection{Examples}\label{E:iredpoly} In this subsection,
we present a family of irreducible polynomials using the sufficient condition 
given in this article.
\smallskip
 
\noindent{\bf Example 1.}
	Let $P_1\in F[z_1,z_2]$ be of the form 
	\[
	az_1^n+bz_2^m + cz_1^uz_2^v + \sum c_{ij}z_1^iz_2^j
	\]
	where $a, b, c, n, m \neq 0$ and $P_2\in F[z_3,\ldots,z_n]$ be any nonconstant polynomial.
	Suppose that $un+mv\neq mn$ and $u+v>\max\{m,\,n\}$ and $u+v>\max\{(i+j):c_{ij}\neq 0\}$.
	Notice that $(u,v)$ will be on the positive 
	side of the line passing through $(n,0)$ and $(0,m)$, i.e., $mu + nv > mn$.
	To see this, suppose first that $n\geq m$. We know $u + v > n$.
	So,
	\[
	mu + nv \geq m(u+v) > mn.
	\]
	The other case could also be verified easily.
	\smallskip
	
	\noindent Suppose now for each $(i,j)$ for which $c_{ij}\neq 0$,
	we have $mi+nj\geq mn$, $vi-(u-n)j \leq vn$, and $uj-(v-m)i \leq mu$.
	Then $\nwtp{P_1}$ will be a triangle with vertices $(n,0), (m,0)$ and $(u,v)$. 
	By Result~\ref{Res:indprism} and Proposition \ref{P:anocritirred}, 
	we see that $P=P_1+P_2$ is irreducible if $\gcd(m,n,u,v)=1$.
	\smallskip
	
	\noindent{\bf Example~2.}
	Suppose we can rearrange the terms of $P\in F[z_1, \ldots, z_n]$ such that we can
	write $P=P_1 + P_2$ where $P_1$ can be written in the form $P_1=Q_1 + Q_2$ 
	where $Q_1\in F[z_1]$ is a nonconstant polynomial of degree $r$ 
	and $Q_2\in F[z_2, \ldots, z_\nu]$ is a nonconstant polynomial of degree $m$ such that $P_1(0)=0$.
	Observe that if $\gcd(r,m)=1$ then the method of the proof
	of Theorem~\ref{T:irredCritdecpoly} implies that $\nwtp{P_1}$ is integrally
	indecomposable polytope. Hence by Proposition~\ref{P:anocritirred}, $P$ will 
	be irreducible for any nonconstant $P_2\in F[z_{\nu+1},\ldots,z_n]$.

\section*{Acknowledgements}
A part of this work was carried out at the Indian Institute of Science (IISc), Bangalore where the
first author was a research associate. He wishes to thank his thesis adviser Prof.\,Gautam Bharali
for supporting him at this position under his Swarnajayanti
Fellowship (Grant No.~DST/SJF/MSA-02/2013-14). 

\end{document}